\documentclass[a4paper]{amsart}

\usepackage[longnamesfirst,comma]{natbib}

\usepackage[english]{babel}
\usepackage[utf8]{inputenc}
\usepackage{lmodern}
\usepackage[T1]{fontenc}

\usepackage{amssymb} % symboles de l'AMS comme \triangleq
\usepackage{version} % comments

\usepackage{eurosym} % Symbole Euro

\usepackage{graphicx}
\def\EuScript{\mathcal}
\def\mathscr{\EuScript}

\newcommand{\defegal}{:=}                                   % D\'{e}finition

\newcommand{\bbR}{\mathbb{R}}                               % Nombres r\'{e}els

\newcommand{\ad}{^{\mathrm{ad}}}                            % Admissibilit\'{e}
\newcommand{\np}[1]{(#1)}                                   % Parenth\`{e}se normal
\newcommand{\Bp}[1]{\Big(#1\Big)}                           % Parenth\`{e}se Big
\newcommand{\bgp}[1]{\bigg(#1\bigg)}                        % Parenth\`{e}se bigg
\newcommand{\Bgp}[1]{\Bigg(#1\Bigg)}                        % Parenth\`{e}se Bigg

\newcommand{\espacea}[1]{\mathbb{#1}}                       % Espace d'arriv\'{e}e
\newcommand{\espacef}[1]{\mathcal{#1}}                      % Espace fonctionnel
\newcommand{\tribu}[1]{\mathscr{#1}}                        % Tribu
\newcommand{\omeg}{\Omega}                                  % espace du triplet
\newcommand{\trib}{\tribu{A}}                               % tribu  du triplet
\newcommand{\prbt}{\mathbb{P}}                              % proba  du triplet
\newcommand{\espe}{\mathbb{E}}                              % Symbole esp\'{e}rance
\newcommand{\va}[1]{\boldsymbol{\uppercase{#1}}}            % Variable al\'{e}atoire
\newcommand{\normdelim}[1]{\np{#1}}                         % Taille ``normal''
\newcommand{\Bigdelim}[1]{\Bp{#1}}                          % Taille ``Big''
\newcommand{\biggdelim}[1]{\bgp{#1}}                        % Taille ``bigg''
\newcommand{\Biggdelim}[1]{\Bgp{#1}}                        % Taille ``Bigg''
\newcommand{\normdelims}[2]{\normdelim{#1\mid#2}}           % avec s\'{e}parateur
\newcommand{\Besp}[2][]{\espe_{#1}\Bigdelim{#2}}            % Esp\'{e}rance Big
\newcommand{\bgesp}[2][]{\espe_{#1}\biggdelim{#2}}          % Esp\'{e}rance bigg
\newcommand{\Bgesp}[2][]{\espe_{#1}\Biggdelim{#2}}          % Esp\'{e}rance Bigg

\newcommand{\nespc}[3][]{\espe_{#1}\normdelims{#2}{#3}}     % Esp\'{e}rance cond. normal
\newcommand{\esper}[2][]{\espe_{#1}\left(#2\right)}           % Esp\'{e}rance avec argument optionnel
\newcommand{\espcond}[3][]{\espe_{#1}\left(#2\ \middle\vert\ #3\right)}    % Esp\'{e}rance cond.
\theoremstyle{plain}
\newtheorem{thm}{Theorem}
\newtheorem{lemma}{Lemma}
\newtheorem{prop}{Proposition}
\theoremstyle{definition}
\newtheorem{dfn}{Definition}

\theoremstyle{remark}
\newtheorem{rem}{Remark}
\newtheorem{hyp}{Assumption}
\newtheorem{example}{Example}

\usepackage[pdftex,hypertexnames=false,colorlinks=true,linkcolor=black,citecolor=black]{hyperref}

\begin{document}

\title{Price decomposition in large-scale stochastic optimal control}

\author[K. Barty]{Kengy Barty}
\address{K. Barty, EDF R\&D, 1 avenue du G\'{e}n\'{e}ral de Gaulle,
         F-92141 Clamart Cedex, France.}
\email{kengy.barty@edf.fr}

\author[P. Carpentier]{Pierre Carpentier}
\address{P. Carpentier, ENSTA ParisTech,
         32 boulevard Victor, 75739 Paris Cedex 15, France.}
\email{pierre.carpentier@ensta-paristech.fr}

\author[G. Cohen]{Guy Cohen}
\address{G. Cohen, Universit\'{e} Paris-Est, CERMICS, \'{E}cole des Ponts ParisTech,
         6 \& 8 avenue Blaise Pascal, 77455 Marne-la-Vall\'{e}e Cedex 2.}
\email{guy.cohen@mail.enpc.fr}

\author[P. Girardeau]{Pierre Girardeau}
\address{P. Girardeau, EDF R\&D, 1 avenue du G\'{e}n\'{e}ral de Gaulle,
         F-92141 Clamart Cedex, France,
         also with Universit\'{e} Paris-Est, CERMICS and ENSTA ParisTech.}
\email{pierre.girardeau@cermics.enpc.fr}

\date{\today}


\begin{abstract}
	We are interested in optimally driving a dynamical system that can be influenced by exogenous noises. This is generally called a Stochastic Optimal Control~(SOC) problem and the Dynamic Programming~(DP) principle is the natural way of solving it. Unfortunately, DP faces the so-called curse of dimensionality: the complexity of solving DP equations grows exponentially with the dimension of the information variable that is sufficient to take optimal decisions~(the state variable).
	
	For a large class of SOC problems, which includes important practical problems, we propose an original way of obtaining strategies to drive the system. The algorithm we introduce is based on Lagrangian relaxation, of which the application to decomposition is well-known in the deterministic framework. However, its application to such closed-loop problems is not straightforward and an additional statistical approximation concerning the dual process is needed. We give a convergence proof, that derives directly from classical results concerning duality in optimization, and enlghten the error made by our approximation. Numerical results are also provided, on a large-scale SOC problem. This idea extends the original DADP algorithm that was presented by \citet{BartyCarpentierGirardeau09}.
\end{abstract}

\maketitle

\section*{Introduction}

Consider a controlled dynamical system over a discrete and finite time horizon. This system may be influenced by exogenous noises that affect its behaviour. We suppose that, at every instant, the decision maker is able to observe these noises and to keep these observations in memory. Since it is generally profitable to take available observations into account when designing future decisions, we are looking for strategies rather than simple decisions. Such strategies~(or policies) are feedback functions that map every instant and every possible history of the system to a decision to be made.

More precisely, we are here interested in optimization problems with a large number of variables. The typical application we have in mind is the following. Consider a power producer that owns a certain number of power units. Each unit has its own local characteristics such as physical constraints that restrain the set of feasible decisions, and production costs that depend on the type of fuel that is used to produce power. The power producer has to control the power units so that a global power demand is met at every instant. The power demand, as well as other parameters such as inflows in water reservoirs or unit breakdowns, are random. Naturally, he is looking for strategies that make the production cost minimal, over a given time horizon. In such a problem, both the number of power units and the number of time steps are usually large.

One classical approach when dealing with stochastic dynamic optimization problems is to discretize the random inputs of the problem using scenario trees. Such an approach has been widely studied within the Stochastic Programming community~\citep[see the book by][for an overview of this methodology]{ShapiroDentchevaRuszczynski09}. One of the advantages of such a technique is that as soon as the scenario tree is drawn, the derived problem can be treated by classical Mathematical Programming techniques. Thus, a number of decomposition methodologies have been proposed~\citep[][Chapter~3]{StochasticDecomposition96, CarpentierCohenCulioliRenaud96,StochasticProgramming03} and even applied to energy planning problems~\citep{BacaudLemarechalRenaudSagastizabal01}. A general theoteric point of view concerning the way to combine the discretization of expectation together with the discretization of information is given by~\citet{TheseBarty}. However, in a multi-stage setting, this methodology suffers from the drawbacks that arise with scenario trees. As it was pointed out by~\citet{Shapiro06}, the number of scenarios needed to achieve a given accuracy grows exponentially with the number of time steps of the problem.

The other natural approach to solve SOC problems is to rely on the Dynamic Programming~(DP) principle~\citep[see][]{Bellman57,BertsekasDP}. The core of the DP approach is the definition of a state variable that is, roughly speaking, the variable that, in conjunction with the time variable, is sufficient to take an optimal decision at every instant. It does not have the drawback of the scenario trees concerning the number of time steps since strategies are, in this context, depending on a state variable whose space dimension usually does not grow with time\footnote{In the case of power management, the state dimension is usually the number of power units.}. However, DP suffers from another drawback which is the so-called \emph{curse of dimensionality}: the complexity of solving the DP equation grows exponentially with the state space dimension. Hence, brutally solving the DP equation is generally intractable when the state space dimension goes beyond several units. Recently, \citet{VezolleVialleWarin09} were able to solve it on a~$10$-state-variables energy management problem, using parallel computation coupled with adequate data distribution.

Another popular idea is to represent the value functions~(solutions of the DP equation) as a linear combination of a priori chosen basis functions~\citep[see among others][Sect.~6.5]{BellmanDreyfus59,BertsekasTsitsiklis96}. This approach, called Approximate Dynamic Programming or often Least-Squares Monte-Carlo, has also become very popular in the context of American option pricing through the work of~\citet{LongstaffSchwartz01options}. This approximation reduces the complexity of solving the DP equation drastically. However, in order to be practically efficient, such an approach requires some a priori information about the problem, in order to define a well suited functional subspace. Indeed, there is no systematic means to choose the basis functions and several choices have been proposed in the literature~\citep{deFariasVanRoy03,TsitsiklisVanRoy96,BouchardWarin10}.

When dealing with large-scale optimization problems, the decomposition/co\-ordi\-nation approach aims at finding a solution to the original problem by iteratively solving smaller-dimensional subproblems. In the deterministic case, several types of decomposition have been proposed (e.g. by prices or by quantities) and unified in a general framework using the Auxiliary Problem Principle by \cite{Cohen80}. In the open-loop stochastic case, i.e. when controls do not rely on any observation, \cite{CohenCulioli90} proposed to take advantage of both decomposition techniques and stochastic gradient algorithms. These techniques have been extended in the closed-loop stochastic case by \cite{BartyRoyStrugarek05}, but so far they fail to provide decomposed state dependent strategies in the Markovian case. This is because a subproblem optimal strategy depends on the state of the whole system, not only on the local state. In other words, decomposition approaches are meant to decompose the control space, namely the range of the strategy, but the numerical complexity of the problems we consider here also arises because of the dimensionality of the state space, that is to say the domain of the strategy.

We here propose a way to use price decomposition within the closed-loop stochastic case. The coupling constraints, namely the constraints preventing the problem from being naturally decomposed, are dualized using a Lagrange multiplier (price). At each iteration, the price decomposition algorithm solves each subproblem using the current price, then uses the solutions to update the price. In the stochastic context, price is a random process whose dynamics is not available, so the subproblems do not in general fall into the Markovian setting. However, in a specific instance of this problem, \cite{TheseStrugarek} exhibited a dynamics for the optimal multiplier, and he showed that these dynamics were independent with respect to the decision variables. Hence it was possible to come down to the Markovian framework and to use DP to solve the subproblems in this case. Following this idea, \citet{BartyCarpentierGirardeau09} proposed to choose a parametrized dynamics for these multipliers in such a way that solving subproblems using DP becomes possible. While the approach, called Dual Approximate Dynamic Programming~(DADP), showed promising results on numerical examples, it suffers from the fact that the induced restrained dual space is non-convex. This led to some numerical instabilities and, probably more important, it was not possible to give convergence results for the algorithm. We here propose to extend DADP in a more general way that allows us to derive convergence results and solves the problem of numerical instabilities.

The paper is organized as follows. In Section~\ref{sec:Math}, we present the general SOC problem and the DP principle. Then we concentrate on a more specific class of problems, that we call decomposable problems, and recall the previous version of the DADP algorithm. In Section~\ref{sec:Sspb}, we present the new version we propose and give convergence results for the algorithm. Finally, in Section~\ref{sec:Num}, we apply DADP to two numerical examples, the first being the one from the previous paper by \citet{BartyCarpentierGirardeau09} and the second one being a more realistic power management example.

\section{Mathematical formulation} \label{sec:Math}

\subsection{General problem setting}

All along the paper, random variables are denoted using \textbf{bold} letters. Consider a discrete and finite time horizon~$0,~1,~\dots,~T$ and a probability space~$(\omeg, \trib, \prbt)$. To define a stochastic dynamical system, we need:
\begin{itemize}
	\item a stock process~$\va{x}=(\va{x}_0, \dots, \va{x}_T)$ which represents the physical states of the system through time, the value of~$\va{x}_t$ lying, at every instant~$t$, in a Hilbert space~$\espacea{X}_t$;
	\item a control process~$\va{u}=(\va{u}_0, \dots, \va{u}_{T-1})$, the value of~$\va{u}_t$ lying, at every instant~$t$, in a Hilbert space~$\espacea{U}_t$;
	\item a noise process~$\va{w}=(\va{w}_0, \dots, \va{w}_{T-1})$, the value of~$\va{w}_t$ lying, at every instant~$t$, in a Hilbert space~$\espacea{W}_t$.
\end{itemize}
The spaces~$\espacea{X}_t$, $\espacea{U}_t$ and~$\espacea{W}_t$ are generally finite-dimensional spaces. In the sequel, we suppose~$\espacea{X}_t=\bbR^n$ and~$\espacea{U}_t=\bbR^m$. The decision variable~$\va{u}_t$ being a random variable, and our purpose being to use variational techniques that require the notion of gradient, it is natural to suppose that~$\va{u}_t$ lies in a Hilbert space~$\espacef{U}_t$, for example~$L^2(\omeg, \trib, \prbt; \espacea{U}_t)$.

The three types of variables are linked together in the following way. At every time step~$t$, there exists a function~$f_t$ (the dynamics of the system) that maps the triplet~$(\va{x}_t, \va{u}_t, \va{w}_t)$ to the next stock value~$\va{x}_{t+1}$. Let~$(\trib_0, \dots, \trib_{T-1})$ be the filtration associated with the stochastic process~$\va{w}$. We suppose that, at every time step~$t$, the decision maker is able to observe and to keep in memory all the past history of~$\va{w}$ up to time~$t$. The causality principle states that the decision~$\va{u}_t$ at time~$t$ is~$\trib_t$-measurable, i.e. only depends on past observations. Moreover, at each time step~$t$, a cost~$C_t(\va{x}_t, \va{u}_t, \va{w}_t)$ is incurred. Finally, at the final time~$T$, a cost~$K(\va{x}_T)$ is added. The Stochastic Optimal Control (SOC) problem we would like to solve hence reads:
\begin{subequations} \label{eqn:P}
\begin{align}
	\min_{\va{x}, \va{u}} \quad& \esper{\sum_{t=0}^{T-1} C_t\left(\va{x}_t, \va{u}_t, \va{w}_t\right) + K\left(\va{x}_T\right)}, \\
\intertext{subject to dynamics constraints:}
	&\va{x}_{t+1} = f_t\left(\va{x}_t, \va{u}_t, \va{w}_t\right), \qquad \forall t=0, \dots,  T-1, \label{eqn:P-2} \\
	&\va{x}_0 \text{ is given}, \\
\intertext{as well as bound constraints:}
	&\underline{x}_t \leq \va{x}_t \leq \overline{x}_t, \qquad \forall t=1, \dots, T,  \label{eqn:P-4} \\
	&\underline{u}_t \leq \va{u}_t \leq \overline{u}_t, \qquad \forall t=0, \dots, T-1,  \label{eqn:P-5} \\
\intertext{static constraints:}
	&g_t\left(\va{x}_t, \va{u}_t, \va{w}_t\right) = 0, \qquad \forall t=0, \dots,  T-1, \label{eqn:P-6}
\intertext{and the non-anticipativity constraint:}
	&\va{u}_t \text{ is } \trib_t\text{-measurable}.
\end{align}
\end{subequations}
Constraints~\eqref{eqn:P-2}, \eqref{eqn:P-4}, \eqref{eqn:P-5} and \eqref{eqn:P-6} have to be understood in the~$\prbt$-almost sure sense. We give examples for constraint~\eqref{eqn:P-6} in~\S\ref{sec:Sspb}. With no further assumptions, Problem~\eqref{eqn:P} cannot generally be solved analytically, except for quite particular cases among which is, for instance, the Linear Quadratic Gaussian (LQG) case. One has to be aware that, when solving this problem, one is looking for functions that map every possible history of the system to a decision; the domain of such a function is clearly growing with time and representing it on a computer rapidly becomes intractable.

\subsection{The Dynamic Programming Principle}

Fortunately enough, control theory helps us reduce the size of the optimal strategy's domain in some cases. Let us first make the following assumption.
\begin{hyp} \label{hyp:Indpdt}
	Noises~$\va{w}_0, \dots, \va{w}_{T-1}$ are independent over time.
\end{hyp}
Now define functions~$V_t$, for every time step~$t=0, \dots, T$, as:
\begin{equation*}
	V_t\left(x\right) = \min_{\shortstack{\scriptsize$\va{x}_t, \dots, \va{x}_T$\\\scriptsize$\va{u}_t, \dots, \va{u}_{T-1}$}} \espcond{\sum_{s=t}^{T-1} C_s\left(\va{x}_s, \va{u}_s, \va{w}_s\right) + K\left(\va{x}_T\right)}{\va{x}_t=x}, \qquad \forall x \in \espacea{X}_t,
\end{equation*}
subject to the same\footnote{while starting at time~$t$} constraints as in Problem~\eqref{eqn:P}. Function~$V_t$ represents the minimal remaining cost of the problem when starting at time~$t$, for every possible stock value~$x$.

Under Assumption~\ref{hyp:Indpdt}, the Dynamic Programming~(DP) principle states that the variable~$\va{x}_t$, along with the current noise value~$\va{w}_t$, contains all the information that is sufficient to take the optimal decision at time~$t$, hence the term \emph{state} variable. Moreover, it provides a way to compute functions~$V_t$, that we now call Bellman functions~(or value functions), as well as optimal strategy, in a backward manner.
\begin{subequations} \label{eqn:DP}
\begin{align}
	V_T\left(x\right) &= K\left(x\right), \qquad \forall x \in \espacea{X}_T, \\
\intertext{and, for every time step~$t=T-1, \dots, 0$:}
	V_t\left(x\right) &= \esper{\min_u C_t\left(x, u, \va{w}_t\right) + V_{t+1}\left(f_t\left(x, u, \va{w}_t\right)\right)}, \qquad \forall x \in \espacea{X}_t.
\end{align}
\end{subequations}
Compared with the original setting where the optimal strategy domain was growing along with time steps, the DP principle drastically reduces the size of the information needed to make an optimal decision.

\begin{rem}[About the overtime independence]
	In the case when the model is such that noises that affect the system have some sort of correlation through time, one can always explicit the dynamics of the noise variable and add it to the dynamics of~$\va{x}_t$, thus defining a new~(albeit larger!) state variable as well as a new noise variable that is now independent over time.
\end{rem}

\begin{rem}[Hazard-Decision setting] \label{rem:hd}
	The reader may have noticed that the way the non-anticipativity constraint in written allows the decision maker at time~$t$ to observe the current noise value~$\va{w}_t$ before choosing the control~$\va{u}_t$. In such a setting the optimal decision at time~$t$ depends on both the state variable~$\va{x}_t$ and the noise variable~$\va{w}_t$ whereas the value function only depends on the state variable~$\va{x}_t$.
\end{rem}

Note however that the dimension of the state space~$\espacea{X}_t$ might still be quite large. Yet the complexity of solving the DP equation~\eqref{eqn:DP} grows exponentially with the dimension of~$\espacea{X}_t$; this unpleasant feature is well known as the \emph{curse of dimensionality} and prevents us from solving this equation by discretization when the state space dimension is, say, greater than~$5$.

\subsection{Decomposable problem setting}

Let us now present a particular instance of Problem~\eqref{eqn:P} on which we are able to reduce even more the size of the information needed to take a reasonable decision.

We consider a system which consists of~$N$ subsystems\footnote{We often use the term ``\emph{units}'' for subsystems.}, whose dynamics and cost functions are independent one from another. More precisely, the state $\va{x}_t$~(respectively the control $\va{u}_t$) of the global system writes $\big(\va{x}_t^1, \dots, \va{x}_t^N\big)$ with $\va{x}_t^i \in L^2\left(\omeg, \trib, \prbt; \bbR^{n_i}\right)$ (resp. $\big(\va{u}_t^1, \dots, \va{u}_t^N\big)$ with $\va{u}_t^i \in L^2\left(\omeg, \trib, \prbt; \bbR^{m_i}\right)$) and $n=\sum_{i=1}^N n_i$ (resp. $m=\sum_{i=1}^N m_i$), so that the global dynamics $\va{x}_{t+1} = f_t\left(\va{x}_t, \va{u}_t, \va{W}_t\right)$ can be written independently unit by unit: $\va{x}_{t+1}^i = f_t^i\left(\va{x}_t^i, \va{u}_t^i, \va{W}_t\right)$, $i=1, \dots, N$. In the same way, the global cost $C_t\left(\va{x}_t, \va{u}_t, \va{W}_t\right)$ is equal to the sum of the local unit costs $C_t^i\left(\va{x}_t^i, \va{u}_t^i, \va{W}_t\right)$, $i=1, \dots, N$. At the end of the time period, each unit~$i$ causes a cost~$K^i$ that only depends on its final state~$\va{x}_T^i$.

Remark that, without further constraints, the induced SOC problem can be stated independently unit by unit, though the same noise variable affects all units (see Appendix~\ref{app:decomposition} for a precise proof). Hence, under Assumption~\ref{hyp:Indpdt}, the solving of the DP equation can be decomposed unit by unit. For each unit, the optimal strategy depends only on its local state\footnote{and on the noise at the current time step because we are in the Hazard-Decision setting}, which is usually far smaller than the dimension of the global state space.

Consider now a static constraint~\eqref{eqn:P-6} that couples the units together. We suppose that such a coupling arises from a set of static $\bbR^d$-valued constraints, the constraint at time step $t$ reading $\sum_{i=1}^N g_t^i\left(\va{x}_t^i, \va{u}_t^i, \va{w}_t\right) = 0$. This kind of coupling constraint is natural in many industrial applications, including the case of a power management problem that we already mentioned in the introduction: the sum of the productions of the power units must meet an uncertain power demand.

The decomposable problem we are interested in solving in the following reads:
\begin{subequations} \label{eqn:Pdec}
\begin{align}
	\min_{\va{x}, \va{u}} \quad& \esper{\sum_{t=0}^{T-1} \sum_{i=1}^N C_t^i\left(\va{x}_t^i, \va{u}_t^i, \va{w}_t\right) + \sum_{i=1}^N K^i\left(\va{x}_T^i\right)} \\
\intertext{subject to dynamics constraints:}
	&\va{x}_{t+1}^i = f_t^i\left(\va{x}_t^i, \va{u}_t^i, \va{w}_t\right), \qquad \forall t=0, \dots,  T-1, \forall i=1, \dots, N, \label{eqn:Pdec-2} \\
	&\va{x}_0^i \text{ is given}, \qquad \forall i=1, \dots, N, \label{eqn:Pdec-3} \\
\intertext{as well as bound constraints:}
	&\underline{x}_t^i \leq \va{x}_t^i \leq \overline{x}_t^i, \qquad \forall t=1, \dots, T, \forall i=1, \dots, N, \\
	&\underline{u}_t^i \leq \va{u}_t^i \leq \overline{u}_t^i, \qquad \forall t=0, \dots, T-1, \forall i=1, \dots, N, \\
\intertext{static constraints:}
	&\sum_{i=1}^N g_t^i\left(\va{x}_t^i, \va{u}_t^i, \va{w}_t\right) = 0, \qquad \forall t=0, \dots,  T-1, \label{eqn:Pdec-6} \\
\intertext{and the non-anticipativity constraint:}
	&\va{u}_t^i \text{ is } \trib_t\text{-measurable}, \qquad \forall t=0, \dots, T-1, \forall i=1, \dots, N. \label{eqn:Pdec-7}
\end{align}
\end{subequations}

There are three types of coupling in Problem~\eqref{eqn:Pdec}:
\begin{itemize}
	\item The first comes from the state dynamics~\eqref{eqn:Pdec-2} that induce a temporal coupling.
	\item The second one arises from the static constraints~\eqref{eqn:Pdec-6} that induce a spatial coupling: they link together all the subsystems at each time step~$t$.
	\item The third type of coupling is informational: it comes from the causality constraint~\eqref{eqn:Pdec-7}, which prevents us from decomposing directly scenario by scenario~: if two realizations of the noise process are identical up to time~$t$, then the same control has to be applied at time~$t$ on both realizations.
\end{itemize}

Constraints~\eqref{eqn:Pdec-6} prevent us from decomposing the optimization problem unit by unit: the solution~$\va{u}_t^i$ for unit~$i$ and time~$t$ has to be searched as a feedback function~$\varphi_t^i$ depending on the current noise value and on the whole stock variable~$\va{x}_t=(\va{x}_t^1, \dots, \va{x}_t^N)$ rather than on the local stock variable~$\va{x}_t^i$! Adding the coupling constraint~\eqref{eqn:Pdec-6} drastically changed the structure of the problem.

\begin{rem}[Local and global noises] \label{rem:fulldec}
	Applications we have in mind are power management problems which are completely ``flower-shaped'', in the following sense. The noise variable~$\va{w}_t$ at time~$t$ is composed of two different kinds of noise:
	\begin{itemize}
		\item a \emph{local} noise~$\va{w}_t^i$ for every subsystem~$i$, i.e. at every petal of the flower~(uncertain inflows entering a water reservoir, for instance);
		\item a \emph{global} noise~$\va{d}_t$ at the center of the flower~(a total power demand, for instance).
	\end{itemize}
	In such a setting, only the local noise appears in the cost function and in the dynamics, leading to functions of the form:
	\begin{equation*}
		C_t^i\left(\va{x}_t^i, \va{u}_t^i, \va{w}_t^i\right) \text{ and } f_t^i\left(\va{x}_t^i, \va{u}_t^i, \va{w}_t^i\right),
	\end{equation*}
	while the global noise appears only in the coupling constraint as, for instance:
	\begin{equation*}
		\sum_{i=1}^N g_t^i\left(\va{x}_t^i, \va{u}_t^i\right) = \va{d}_t.
	\end{equation*}
	Keeping this particular case in mind shall give us some insight about how to decompose the global problem as well as possible. This is explained in more details in~\S\ref{ssec:proj} and such settings are treated in the numerical experiments of~\S\ref{sec:Num}.
\end{rem}

\subsection{Previous paper} \label{ssec:previous}

In a previous study~\citep{BartyCarpentierGirardeau09}, the authors proposed a way of handling Problem~\eqref{eqn:Pdec} by approximate Lagrangian decomposition. The proposed algorithm, called Dual Approximate Dynamic Programming~(DADP) is as follows. Let us introduce the Lagrangian of Problem~\eqref{eqn:Pdec}:
\begin{multline*}
	\mathcal{L}\left(\va{x}, \va{u}, \va{\lambda}\right) \defegal \Bgesp{\sum_{t=0}^{T-1} \sum_{i=1}^N \left(C_t^i\left(\va{x}_t^i, \va{u}_t^i, \va{w}_t\right) + \va{\lambda}_t^\top g_t^i\left(\va{x}_t^i, \va{u}_t^i, \va{w}_t\right)\right) \\
	+ \sum_{i=1}^N K^i\left(\va{x}_T^i\right)},
\end{multline*}
with~$\va{\lambda}_t \in L^2(\omeg, \trib, \prbt; \bbR^d)$ the Lagrange multiplier of the coupling constraint~\eqref{eqn:Pdec-6} and~$\va{\lambda} \defegal (\va{\lambda}_0, \dots, \va{\lambda}_{T-1})$. Note that, since the dualized constraint is~$\trib_t$-measurable, the Lagrange multiplier~$\va{\lambda}_t$ need only to have the same measurability.

Problem~\eqref{eqn:Pdec} is always equivalent to:
\begin{equation*}
	\min_{\va{x}, \va{u}} \max_{\va{\lambda}} \quad \mathcal{L}\left(\va{x}, \va{u}, \va{\lambda}\right),
\end{equation*}
where the minimization is subject to all constraints of Problem~\eqref{eqn:Pdec} except constraint~\eqref{eqn:Pdec-6}. If~$\mathcal{L}$ has a saddle point~(see Appendix~\ref{app:uzawa} for a definition and a characterization of saddle points), then this problem is equivalent to the so-called dual problem:
\begin{equation} \label{eqn:D}
	\max_{\va{\lambda}} \min_{\va{x}, \va{u}} \quad \mathcal{L}\left(\va{x}, \va{u}, \va{\lambda}\right),
\end{equation}
under, once again, the same constraints as in Problem~\eqref{eqn:Pdec} except the coupling constraint~\eqref{eqn:Pdec-6}.

The key point of the so-called price decomposition algorithm is that the inner minimization problem can be split into~$N$ subproblems, each one involving a single subsystem~(once again, see Appendix~\ref{app:decomposition} for more details). One might think that solving these subproblems is much simpler than solving the original global problem. This is not the case here: because the dual variable~$\va{\lambda}$ is a stochastic process that depends in general on the whole history of the system, we cannot reasonably make the overtime independence assumption that leads to the DP principle and subproblems are just as hard as Problem~\eqref{eqn:P}!

The idea of \citet{BartyCarpentierGirardeau09} is to force the dual process to satisfy a prescribed dynamics:
\begin{subequations} \label{eqn:DynLambda}
\begin{align}
	\va{\lambda}_0 &= h_{\alpha_0}\left(\va{w}_0\right), \\
	\va{\lambda}_{t+1} &= h_{\alpha_{t+1}}\left(\va{\lambda}_t, \va{w}_{t+1}\right), \qquad \forall t=0, \dots, T-2,
\end{align}
\end{subequations}
where~$h_{\alpha_t}$ is an a priori chosen function parametrized by~$\alpha_t \in \bbR^q$. We note~$\alpha=(\alpha_0, \dots, \alpha_{T-1})$. Given a vector~$\alpha^k$ of coefficients at iteration~$k$ of the algorithm which defines the current values of the dual variables, the first step of DADP is to solve the~$N$ subproblems by DP with state~$(\va{x}_t^i, \va{\lambda}_t)$. In order to update the Lagrange multipliers, the authors propose to draw~$S$ trajectory samples of the noise~$\va{w}$ and integrate the dynamics~\eqref{eqn:Pdec-2}--\eqref{eqn:Pdec-3} and~\eqref{eqn:DynLambda} using the optimal feedback laws obtained at the first step, thus obtaining~$S$ sample trajectories of~$\va{x}^k$, $\va{u}^k$ and~$\va{\lambda}^k$. A gradient step is then performed sample by sample:
\begin{equation*}
	\va{\lambda}_t^{k+\frac{1}{2},s} = \va{\lambda}_t^{k,s} + \rho_t \times \sum_{i=1}^N g_t^i\left(\va{x}_t^{i,k,s}, \va{u}_t^{i,k,s}, \va{w}_t^{s}\right), \qquad \forall s=1, \dots, S,
\end{equation*}
with~$\rho_t$ obeying the rules of the step-size choice in Uzawa's algorithm~(see Appendix~\ref{app:uzawa}). Finally, we solve the following regression problem:
\begin{multline*}
	\min_{\alpha_0, \dots, \alpha_{T-1}} \sum_{s=1}^S \bigg( \left\Vert h_{\alpha_0}\left(\va{w}_0^s\right) - \va{\lambda}_0^{k+\frac{1}{2},s} \right\Vert_{\bbR^d}^2 \\
	+\sum_{t=0}^{T-2} \left\Vert h_{\alpha_{t+1}}\left(\va{\lambda}_t^{k+\frac{1}{2},s}, \va{w}_{t+1}^s\right) - \va{\lambda}_{t+1}^{k+\frac{1}{2},s} \right\Vert_{\bbR^d}^2 \bigg).
\end{multline*}
The last minimization produces coefficients~$\alpha^{k+1}$ which define, using Equation~\eqref{eqn:DynLambda}, a new process~$\va{\lambda}^{k+1}$.

This procedure has several advantages, notably that its complexity is linear with respect to the number~$N$ of subproblems and that it may lead, depending on the choice for the dual dynamics~$h$, to tractable approximations of the original problem. The authors illustrate this fact on a small example on which they are able to compare standard DP and DADP.

Still, it has some drawbacks, mainly theoretical. First of all, the shape of the dynamics introduced for the dual process is arbitrarily and once for all chosen and the quality of the result depends on this choice. Moreover, this dynamics defines a subspace which is non-convex. The next iterate~$\va{\lambda}^{k+1}$ being a projection on this subspace, it is not well defined and some oscillations observed in practice may be due to this fact. Finally, this non-convexity prevents us from obtaining convergence results for this algorithm.

\section{Dual Approximate Dynamic Programming revisited} \label{sec:Sspb}

We now propose a new version of the DADP algorithm and show how it overcomes the above mentioned drawbacks encountered with the original algorithm. In this new approach, we do not suppose a given dynamics for the multipliers anymore. Still, we use the standard price decomposition algorithm and perform the update of the multipliers scenario-wise using the classical gradient step:
\begin{equation*}
	\va{\lambda}_t^{k+1,s} = \va{\lambda}_t^{k,s} + \rho_t \times \sum_{i=1}^N g_t^i\left(\va{x}_t^{i,k,s}, \va{u}_t^{i,k,s}, \va{w}_t^{s}\right), \qquad \forall s=1, \dots, S.
\end{equation*}
The difficulty is now to solve the subproblems, as explained in~\S\ref{ssec:proj}.

\subsection{Projection of the dual process} \label{ssec:proj}

After Lagrangian decomposition of Problem~\eqref{eqn:Pdec} with a given multiplier~$\va{\lambda}$, the~$i$-th subproblem reads:
\begin{subequations} \label{eqn:subP}
\begin{align}
	\min_{\va{x}^i, \va{u}^i} \quad& \esper{\sum_{t=0}^{T-1} \left(C_t^i\left(\va{x}_t^i, \va{u}_t^i, \va{w}_t\right) + \va{\lambda}_t^\top g_t^i\left(\va{x}_t^i, \va{u}_t^i, \va{w}_t\right)\right) + K^i\left(\va{x}_T^i\right)} \\
\intertext{subject to dynamic constraints:}
	&\va{x}_{t+1}^i = f_t^i\left(\va{x}_t^i, \va{u}_t^i, \va{w}_t\right), \qquad \forall t=0, \dots,  T-1, \label{eqn:subP-2} \\
	&\va{x}_0^i \text{ is given}, \\
\intertext{as well as bound constraints:}
	&\underline{x}_t^i \leq \va{x}_t^i \leq \overline{x}_t^i, \qquad \forall t=1, \dots, T, \\
	&\underline{u}_t^i \leq \va{u}_t^i \leq \overline{u}_t^i, \qquad \forall t=0, \dots, T-1, \\
\intertext{and the non-anticipativity constraint:}
	&\va{u}_t^i \text{ is } \trib_t\text{-measurable}. \label{eqn:subP-6}
\end{align}
\end{subequations}

As it was already mentioned, since the dual stochastic process~$\va{\lambda}$ generally depends on the whole history of the process, solving this problem is in general as complex as solving the original problem. In order to bypass this difficulty, let us choose at each time step~$t$ a random variable~$\va{y}_t^i$ that is measurable with respect to~$\trib_t$. We call~$\va{y}^i = (\va{y}_0^i, \dots, \va{y}_{T-1}^i)$ the information process for subsystem~$i$. The idea is to rely on a short memory process~$\va{y}^i$. Note that we require that this random process is not influenced by controls. We propose to replace Problem~\eqref{eqn:subP} by:
\begin{equation} \label{eqn:subPapprox}
	\min_{\va{x}^i, \va{u}^i} \, \esper{\sum_{t=0}^{T-1} \left(C_t^i\left(\va{x}_t^i, \va{u}_t^i, \va{w}_t\right) + \espcond{\va{\lambda}_t}{\va{y}_t^i}^\top g_t^i\left(\va{x}_t^i, \va{u}_t^i, \va{w}_t\right)\right) + K^i\left(\va{x}_T^i\right)},
\end{equation}
subject to constraints~\eqref{eqn:subP-2}--\eqref{eqn:subP-6}.

Let us first examine the special situation in which the information variable~$\va{y}_t^i$ only depends on the current noise~$\va{w}_t$. The process~$\va{y}^i$ does not add memory in the system so that Problem~\eqref{eqn:subPapprox} can be solved using the standard DP equation:
\begin{align*}
	V_T^i\left(x\right) &= K^i\left(x\right), \qquad \forall x \in \espacea{X}_T^i, \\
	V_t^i\left(x\right) &= \Besp{\min_{u \in \espacea{U}^i} C_t^i\left(x, u, \va{w}_t\right) + \espcond{\va{\lambda}_t}{\va{y}_t^i}^\top g_t^i\left(x, u, \va{w}_t\right)\\
	&\hspace{165pt} + V_{t+1}^i\left(f_t^i\left(x,u,\va{w}_t\right)\right)}, \qquad \forall x \in \espacea{X}_t^i.
\end{align*}
The expectation quadrature only involves the noise variable~$\va{w}_t$. Remember, as explained in Remark~\ref{rem:hd}, that we are in the ``hazard-decision'' setting: even though the control at each instant~$t$ depends on both~$\va{x}_t^i$ and~$\va{w}_t$, the Bellman function only depends on~$\va{x}_t^i$.

Because of the overtime independence of the information variables~$\va{y}_t^i$, we have to solve DP equations whose dimension is the subsystem dimension~$n_i$. Let us give three examples of choices for~$\va{y}_t^i$.

\begin{example}[Maximal information]
	One can choose to include in~$\va{y}_t^i$ all the noise at time~$t$. As already explained in Remark~\ref{rem:fulldec}, the cost function and dynamics of a subsystem may only depend on a part of the whole noise~$\va{w}_t$~(a kind of \emph{local} information denoted by~$\va{w}_t^i$ in Remark~\ref{rem:fulldec}). Yet some \emph{global} noise, denoted by~$\va{d}_t$ in Remark~\ref{rem:fulldec} may appear in the coupling constraint~(e.g. a global power demand). Hence this maximal choice for the information variable makes the multiplier depend on both local and global information: this shall improve the subsystem's vision of the rest of the system and hence improves strategies. Note, however, that including all the noise at time~$t$ in the information variable is only possible in practice when the noise dimension is not too large. Indeed, the information variable appears in a conditional expectation, whose computation is subject to the curse of dimensionality.
\end{example}

\begin{example}[Minimal information] \label{ex:mininfo}
	On the opposite, one can choose~$\va{y}_t^i=0$ or any other constant. The dual stochastic process is then approximated by its expectation at every instant. Compared to the previous example, there is no conditional expectation anymore but one obtains a strategy that corresponds to the vision of an average price.
\end{example}

\begin{example}[In between]
	One can choose~$\va{y}_t^i$ of the form~$h_t^i(\va{w}_t)$. In practice, this choice will be guided by the intuition one has on which information mostly ``explains'' the optimal price of the system. One has to make a compromise between sufficient information to take reasonable actions and a not too large information variable to be able to compute the conditional expectation in~\eqref{eqn:subPapprox}.
\end{example}

Let us move towards the general case where one can choose to keep some information in memory. In other words, one can choose an  information variable that has a Markovian dynamics, i.e. of the form~$\va{y}_{t+1}^i = h_t^i(\va{y}_t^i, \va{w}_{t+1})$. In order to derive a DP equation in this case, one has to augment the state vector by embedding~$\va{y}_t^i$, that is the necessary memory to compute the next information variable. Thus, the Bellman function associated with the~$i$-th subproblem depends, at time~$t$, on both~$\va{x}_t^i$ and~$\va{y}_{t-1}^i$. The DP equation writes:
\begin{align*}
	V_t^i\left(x, y\right) = \Besp{\min_u \quad& C_t^i\left(x, u, \va{w}_t\right) + \espcond{\va{\lambda}_t^\top}{\va{y}_t^i} \cdot g_t^i\left(x,u,\va{w}_t\right) \\
	&\hspace{5cm}+ V_{t+1}^i\left(f_t^i\left(x, u, \va{w}_t\right), \va{y}_t^i\right)}, \\
	\text{with} \quad& \va{y}_t^i = h_{t-1}^i\left(y, \va{w}_t\right).
\end{align*}
When solving this equation, one obtains controls as feedback functions on the local stock~$\va{x}_t^i$, the current noise~$\va{w}_t$ and the information variable~$\va{y}_{t-1}^i$ of the previous time step. The index gap between information and stock variables comes from the ``hazard-decision'' setting: at time~$t$, the information that is used to take decisions is the conjunction of the information kept in memory~(that has index~$t-1$) and of the noise observed at the current time step~$\va{w}_t$. The sketch of the DADP algorithm is depicted in Figure~\ref{fig:decomposition-schema}.
\begin{figure}[tb]
\begin{center}
	\includegraphics{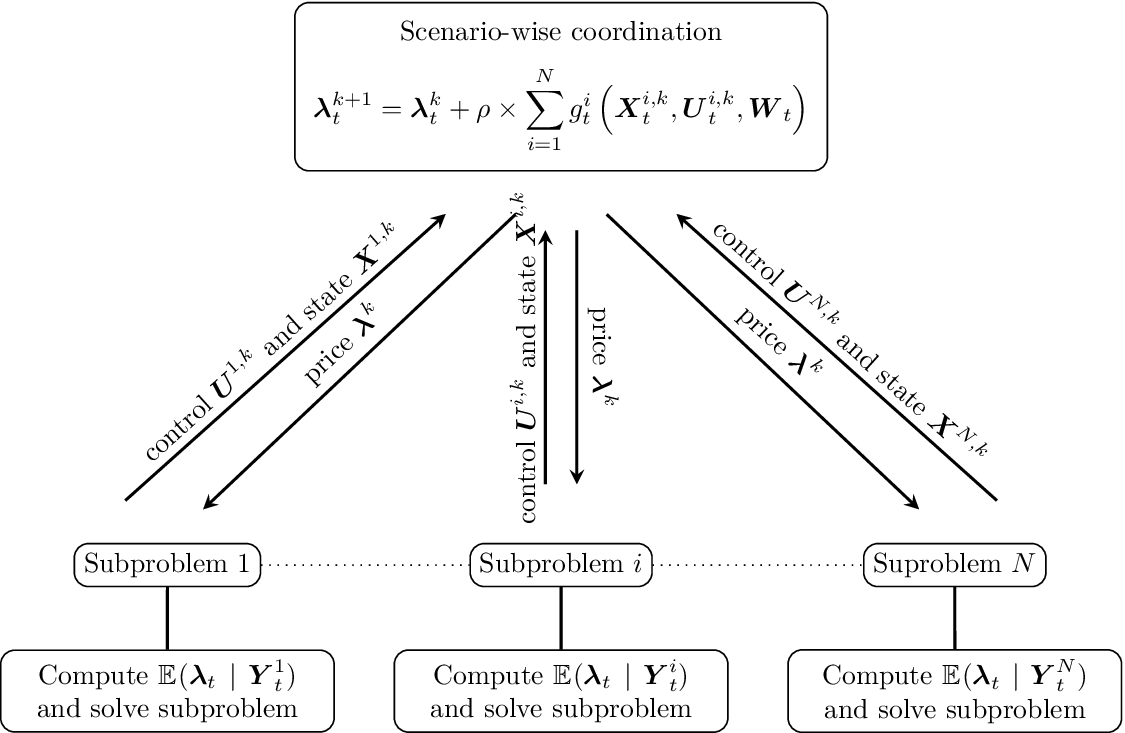}
\end{center}
\caption{\label{fig:decomposition-schema}Dual Approximate Dynamic Programming}
\end{figure}

\begin{example}[Perfect memory] \label{ex:maxinfo}
	The choice~$\va{y}_t^i = \left(\va{w}_0, \dots, \va{w}_t\right)$ stands in the Markovian case. We have then~$\espcond{\va{\lambda}_t}{\va{y}_t^i} = \va{\lambda}_t$. This choice hence allows us to model the dual variable perfectly, but the induced DP equation is unsolvable in practice.
\end{example}

\begin{example}[\citealp{TheseStrugarek}]
	In his PhD thesis, Strugarek exhibited a case when an exact model for the dual process can be obtained. His example is inspired from the kind of power management problem mentioned in the introduction, where~$N$ water reservoirs have to contribute to a global power demand, the rest of this demand being produced by fossil fuel. The noise at each time step~$t$ is composed of a scalar inflow~$\va{a}_t^i$ for each reservoir~$i=1, \dots, N$, and of a scalar power demand~$\va{d}_t$. The problem reads:
	\begin{subequations} \label{eqn:Pstrug}
	\begin{align}
		\min_{\va{x}, \va{u}} \quad& \esper{\sum_{t=1}^{T-1} \sum_{j=1}^n c_j \frac{\left(\va{u}_t^j\right)^2}{2} + \frac{\gamma_j}{2} \left(\va{x}_t^j - x_1^j\right)^2}, \\
	\intertext{where~$c_j$, $j=1, \dots, N$ and~$\gamma_j$, $j=1, \dots, N$ are given real values, subject to dynamic constraints on reservoirs:}
		& \va{x}_{t+1}^j = \va{x}_t^j + \va{a}_{t+1}^j - \va{u}_t^j, \qquad \forall t=1, \dots, T-1, \forall j=1, \dots, n, \\
	\intertext{the power demand constraint:}
		& \sum_{j=1}^n \va{u}_t^j = \va{d}_t, \qquad \forall t=1, \dots, T-1, \label{eqn:Pstrug-1} \\
	\intertext{and the non-anticipativity constraint:}
		& \va{u}_t \text{ is } \sigma\left\{\va{d}_s, s \leq t~; \va{a}_s, s \leq t\right\}\text{-measurable}.
	\end{align}
	\end{subequations}
	
	Let us denote~$\va{a}_t^\sigma \defegal \sum_{i=1}^N \va{a}_t^i$. The author then shows the following result.
	\begin{prop}[\citealp{TheseStrugarek}, Chapter~V] \label{prop:strug}
		If random variables~$(\va{d}_t, \va{a}_t)_{t=1, \dots, T}$ are independent over time, and if there exists~$\alpha>0$ such that~$\gamma_j=\alpha c_j$, for all $j=1, \dots, n$, then the optimal multiplier~$\va{\lambda}$ associated with the coupling constraints~\eqref{eqn:Pstrug-1} satisfies the following dynamics:
		\begin{align*}
			\va{\lambda}_1 &= \frac{1}{\sum_{j=1}^n \frac{1}{c_j}} \left(\va{d}_1 \left(1-\alpha\right) - \alpha \sum_{s=2}^T \esper{\va{a}_s^\sigma} - \alpha \sum_{s=2}^{T-1} \esper{\va{d}_s} \right), \\
			\va{\lambda}_{t+1} &= \va{\lambda}_t + \frac{1}{\sum_{i=1}^n \frac{1}{c_i}} \Big[\va{d}_{t+1} \left(1+\alpha\right) - \va{d}_t - \alpha \esper{\va{d}_{t+1}} \\
			&\hspace{127pt} - \alpha \left(\va{a}_{t+1}^\sigma-\esper{\va{a}_{t+1}^\sigma}\right)\Big], \qquad \forall t=1, \dots, T-2.
		\end{align*}
	\end{prop}
	This allows the solving of subproblems using DP in dimension~3. Note that this example enters our approach if one chooses~$(\va{y}_t, \va{d}_t)$ as an information variable, with:
	\begin{align*}
		\va{y}_1 &= \frac{1}{\sum_{i=1}^n \frac{1}{c_i}} \left(\va{d}_1 \left(1-\alpha\right) - \alpha \sum_{s=2}^T \esper{\va{a}_s^\sigma} - \alpha \sum_{s=2}^{T-1} \esper{\va{d}_s} \right), \\
	\intertext{and, for all~$t=1, \dots, T-2$:}
		\va{y}_{t+1} &= \va{y}_t + \frac{1}{\sum_{i=1}^n \frac{1}{c_i}} \Big[\va{d}_{t+1} \left(1+\alpha\right) - \va{d}_t - \alpha \esper{\va{d}_{t+1}} - \alpha \left(\va{a}_{t+1}^\sigma-\esper{\va{a}_{t+1}^\sigma}\right)\Big].
	\end{align*}
	We get back to the particular case when~$\nespc{\va{\lambda}_t}{\va{y}_t^i} = \va{\lambda}_t$, with a small dimensional information variable~$\va{y}_t^i$. Note however that conditions of Proposition~\ref{prop:strug}, especially the proportionality relation on costs, make little sense in practice.
\end{example}

%\subsection{The DADP algorithm}

\subsection{Convergence} \label{ssec:convergence}

We now give convergence results about DADP and explain in more details the relation between the strategies it builds and the solution of the original problem~\eqref{eqn:P}. To make the paper self-contained, we recall in Appendix~\ref{app:uzawa} the general results concerning duality in optimization, of which the properties of DADP are direct consequences.

The approximation made on the dual process gives us a tractable way of computing strategies for each one of the subsystems. Depending on the choice we make for the information variable, it is quite clear that some strategies will lead to better results than others, concerning the value of the dual problem or the satisfaction of the coupling constraint. Let us here state more precisely these facts.

From now on, we consider a unique information variable for all subsystems. We denote it by~$\va{y}_t$ and define Hilbert spaces
\begin{equation*}
	\mathcal{Y}_t \defegal \{\va{\lambda}_t \in L^2(\omeg, \trib, \prbt): \va{\lambda}_t \text{ is } \va{y}_t\text{-measurable}\},
\end{equation*}
for every~$t=0, \dots, T-1$.
\begin{prop} \label{prop:solve}
	Consider the following optimization problem:
	\begin{subequations} \label{eqn:Pprime}
	\begin{align}
		\min_{\va{x}, \va{u}} \quad& \esper{\sum_{t=0}^{T-1} \sum_{i=1}^N C_t^i\left(\va{x}_t^i, \va{u}_t^i, \va{w}_t\right) + \sum_{i=1}^N K^i\left(\va{x}_T^i\right)}, \label{eqn:Pprime-1} \\
		\intertext{subject to the same constraints as in Problem~\eqref{eqn:Pdec} except the coupling constraint~\eqref{eqn:Pdec-7} which is replaced by:}
		&\espcond{\sum_{i=1}^N g_t^i\left(\va{x}_t^i, \va{u}_t^i, \va{w}_t\right)}{\va{y}_t} = 0, \qquad \forall t=0, \dots, T, \label{eqn:Pprime-4}.
	\end{align}
	\end{subequations}
	Suppose the Lagrangian associated with Problem~\eqref{eqn:Pprime} has a saddle point. Then DADP solves Problem~\eqref{eqn:Pprime}.
\end{prop}
\begin{proof}
	The DADP algorithm consists in:
	\begin{itemize}
		\item given a price process, solving subproblems using the projection of this price process on~$\mathcal{Y}_0 \times \cdots \times \mathcal{Y}_{T-1}$;
		\item updating the price process using a gradient formula.
	\end{itemize}
Alternatively, one may consider that the gradient formula is composed with the projection operation in the updating formula. Therefore, this algorithm may also be viewed as a projected gradient algorithm which exactly solves the following max-min problem~:
	\begin{subequations} \label{eqn:Dapprox}
	\begin{align}
		\max_{\va{\lambda}} \min_{\va{x}, \va{u}} \quad& \esper{\sum_{t=0}^T \sum_{i=1}^N \left(C_t^i\left(\va{x}_t^i, \va{u}_t^i, \va{w}_t\right) + \va{\lambda}_t^\top g_t^i\left(\va{x}_t^i, \va{u}_t^i, \va{w}_t\right)\right) + \sum_{i=1}^N K^i\left(\va{x}_T^i\right)}, \label{eqn:Dapprox-1} \\
		\text{s.t.} \quad& \va{x}_{t+1}^i = f_t^i\left(\va{x}_t^i, \va{u}_t^i, \va{w}_t\right), \qquad \forall t=0, \dots, T-1, \forall i=1, \dots, N, \label{eqn:Dapprox-2} \\
		& \va{x}_0 = \va{w}_0, \label{eqn:Dapprox-3} \\
		&\underline{x}_t^i \leq \va{x}_t^i \leq \overline{x}_t^i, \qquad \forall t=1, \dots, T, \forall i=1, \dots, N, \\
		&\underline{u}_t^i \leq \va{u}_t^i \leq \overline{u}_t^i, \qquad \forall t=0, \dots, T-1, \forall i=1, \dots, N, \\
		& \va{u}_t \text{ is } \trib_t\text{-measurable}, \qquad \forall t=0, \dots, T, \label{eqn:Dapprox-4} \\
		& \va{\lambda}_t \text{ is } \va{y}_t\text{-measurable}, \qquad \forall t=0, \dots, T. \label{eqn:Dapprox-5}
	\end{align}
	\end{subequations}
Observe that the max operation is restricted to a linear subspace defined by \eqref{eqn:Dapprox-5}.
	
	Now, if within the inner product~$\langle \pmb{a}, \pmb{b} \rangle = \esper{\pmb{a}^\top \pmb{b}}$, the variable~$\pmb{a}$ belongs to a given subspace, then the component of~$\pmb{b}$ which is orthogonal to that subspace yields~$0$ in the inner product. Hence it is useless. Put in our context, the multiplier~$\va{\lambda}_t$ can only control the part of~$g_t^i\left(\va{x}_t^i, \va{u}_t^i, \va{w}_t\right)$ which has the same measurability as~$\va{\lambda}_t$. Thus, assuming the existence of a saddle point, that is, the max and min operations can be interchanged in Problem~~\eqref{eqn:Dapprox}, this problem appears as the dual counterpart of Problem~\eqref{eqn:Pprime}.
\end{proof}

Loosely speaking, DADP somehow consists in replacing an almost-sure constraint by a constraint involving a conditional expectation with respect to a so-called information variable. So it is once again clear that if we choose the information variable~$\va{y}_t$ to be the whole history of the system, then we come back to the initial constraint and we in fact solve the original problem. This is the case of Example~\ref{ex:maxinfo}. On the contrary, putting no information at all in~$\va{y}_t$ is the same as satisfying the coupling constraint only in expectation. This is the case of Example~\ref{ex:mininfo}. Note however that it is generally a poor way of representing an almost-sure constraint.

The main difficulty is to find the information variable~$\va{y}_t$ that is going to satisfy the coupling constraint in a fairly good way while keeping the solving process of the subproblems tractable.

We now state the convergence of the DADP algorithm. Let us introduce the objective function~$J: \espacef{U}_0 \times \dots \times \espacef{U}_{T-1} \rightarrow \bbR$ associated with strategy~$\va{u}$, i.e.:
\begin{align*}
	J: \va{U} \mapsto \quad &\esper{\sum_{t=0}^{T-1} \sum_{i=1}^N C_t^i\left(\va{x}_t^i, \va{u}_t^i, \va{w}_t\right) + \sum_{i=1}^N K^i\left(\va{x}_T^i\right)}, \\
	&\text{with: } \va{x}_0 = \va{w}_0, \\
	&\text{and: } \va{x}_{t+1}^i = f_t^i\left(\va{x}_t^i, \va{u}_t^i, \va{w}_t\right), \qquad \forall t=0, \dots, T-1, \forall i=1, \dots, N.
\end{align*}
\begin{prop} \label{prop:convergence}
	If:
	\begin{enumerate}
		\item $J$ is convex, lower semi-continuous, G\^ateaux differentiable,
		\item $J$ is $\alpha$-strongly convex,
		\item all $g_t^i$ are linear and $c$-Lipschitz continuous,
		\item the Lagrangian associated with Problem~\eqref{eqn:Pprime} has a saddle point~$(\overline{\va{u}}, \overline{\va{\lambda}})$,
		\item the step-size~$\rho$ of the algorithm is such that~$0 < \rho < 2 \frac{\alpha}{c^2}$,
	\end{enumerate}
	Then:
	\begin{enumerate}
		\item there exists a unique solution~$\overline{\va{u}}$ of Problem~\eqref{eqn:Pprime},
		\item DADP converges in the sense that~:
		\begin{equation*}
			\va{u}^k \underset{k\to+\infty}{\longrightarrow} \overline{\va{u}} \text{ in } \espacef{U}_0 \times \dots \times \espacef{U}_{T-1},
		\end{equation*}
		\item the sequence~$(\va{\lambda}^k)_{k \geq 0}$ is bounded and every cluster point~$\overline{\va{\lambda}}$ in the weak topology is such that~$(\overline{\va{u}}, \overline{\va{\lambda}})$ is a saddle point of the Lagrangian associated with Problem~\eqref{eqn:Pprime}.
	\end{enumerate}
\end{prop}

\begin{proof}
	The convergence of the algorithm is then a direct application of Theorem~\ref{thm:cohen}, Appendix~\ref{app:uzawa}.
\end{proof}
Note that assumptions of Proposition~\ref{prop:convergence} plus the qualification of constraint~\eqref{eqn:Pprime-4} ensure that the Lagrangian associated with Problem~\eqref{eqn:Pprime} has a saddle point.

\section{Numerical experiment} \label{sec:Num}

We now show the efficiency of DADP on two numerical examples. The first one comes from a previous paper~\citep{BartyCarpentierGirardeau09} in which the authors developped a preliminary version of DADP~(see~\S\ref{ssec:previous}). We show in~\S\ref{ssec:expe1} the good performance of the new version of DADP. The second one, in~\S\ref{ssec:expe2}, is an application to a more realistic power management problem.

\subsection{Computing conditional expectations}

Within the DADP procedure, at each iteration, we have to compute conditional expectations in the criteria~\eqref{eqn:subPapprox} of the subproblems. In order to compute these conditional expectations, we used Generalized Additive Models~(GAMs), that were introduced by~\citet{HastieTibshirani}. The estimate takes the form:
\begin{equation*}
	\espcond{\va{z}}{\va{p}_1, \dots, \va{p}_n} \simeq \sum_{i=1}^n f_i\left(\va{p}_i\right).
\end{equation*}
Functions~$f_i$ are splines~(piecewise polynoms) whose characteristics are optimized by cross-validation on the input statistical data. Our purpose here is not to explain in details this methodology. The interested reader will find further explanations about this model and its implementation in the book by~\citet{Wood}. We used an easy-to-use implementation that is available within the free statistical software R~\citep[][]{R}. The GAM toolkit, called \emph{mgcv}, also returns useful indicators concerning the quality of the estimation. In particular, we use the deviance indicator, which takes value~$0$ if~$\va{z}$ is estimated as poorly as by its expectation~$\esper{\va{z}}$ and value~$1$ if the estimate is exact, i.e. if~$\sum_{i=1}^n f_i(\va{p}_i) = \va{z}$.

\begin{rem}[Kernel estimator]
	We chose to use GAMs to compute conditional expectations after a numerical comparison with the more classical kernel regression methods~\citep{Nadaraya64, Watson64} also available in the~R environment. Even though both of them gave similar results, GAMs appeared to be several times faster than the kernel method on our problem.
\end{rem}

\subsection{Back to an example from a previous paper} \label{ssec:expe1}

We first implement the new version of DADP algorithm on a simple power management problem introduced by \citet{BartyCarpentierGirardeau09}. On this small-scale example, we are able to compare DADP results to those obtained by DP and to illustrate the theoretical results described above. Let us first recall this example. Consider a power producer who owns two types of power plants:
\begin{itemize}
	\item Two hydraulic plants that are characterized at each time step $t$ by their water stock $\va{x}_t^i$ and power production $\va{u}_t^i$, and receive water inflows $\va{a}_{t+1}^i$, $i=1, 2$. Such units are usually cost-free. We however impose small quadratic costs on the hydraulic power productions in order to ensure strong convexity.
	\item One thermal unit with a production cost that is quadratic with respect to its production $\va{u}_t^3$. There are no dynamics associated with this unit.
\end{itemize}
Using these plants, the power producer must supply a power demand $\va{d}_t$ at each time step $t$, over a discrete time horizon of $T=25$ time steps. All noises, i.e. demand $\va{d}_t$ and inflows $\va{a}_t^1$ and $\va{a}_t^2$ are supposed to be overtime independent noise processes. The interested reader may find more details on this numerical experiment in the previous paper by \citet{BartyCarpentierGirardeau09}.

The problem reads:
\begin{subequations} \label{eqn:PTest}
\begin{align}
	\min_{\va{x}, \va{u}} \quad &\esper{\sum_{t=0}^{T-1} \left(\epsilon \left(\va{u}_t^1\right)^2 + \epsilon \left(\va{u}_t^2\right)^2 + L_t\left(\va{u}_t^3\right)\right) + K^1\left(\va{x}_T^1\right) + K^2\left(\va{x}_T^2\right)} \\
	\text{s.t.} \quad &\va{x}_{t+1}^i = \va{x}_t^i - \va{u}_t^i + \va{a}_{t+1}^i, \qquad \forall i=1,2, \quad \forall t=0,\dots, T-1, \\
	& \va{u}_t^1 + \va{u}_t^2 + \va{u}_t^3 = \va{d}_t, \qquad \forall t=0, \dots, T-1, \label{eqn:PTest3} \\
	&\underline{x}^i \leq \va{x}_t^i \leq \overline{x}^i, \qquad \forall i=1,2, \quad \forall t=1,\dots, T, \\
	&0 \leq \va{u}_t^i \leq \overline{u}^i, \qquad \forall i=1,2, \quad \forall t=0,\dots, T-1, \\
	&0 \leq \va{u}_t^3, \qquad \forall t=0, \dots, T-1, \\
	& \va{u}_t^i \text{ is } \sigma\big\{\va{d}_0, \va{a}_0^1, \va{a}_0^2, \dots, \va{d}_t, \va{a}_t^1, \va{a}_t^2\big\} \text{-measurable}, \quad \forall i=1,2,3.
\end{align}
\end{subequations}

In this problem, the state $\va{x}_t$ is two-dimensional, hence DP remains numerically tractable and we can use the DP solution as a reference. In order to use DADP, we choose an information variable~$\va{y}_t$ at time~$t$ that is equal to the power demand~$\va{d}_t$. This comes from the insight that the power demand is a ``global'' information and has all reasons to be useful to the subproblems.

\begin{rem}[Primal feasibility] \label{rem:feasibility}
	In order to validate the method, it has to be evaluated within a simulation procedure. For the evaluation to be fair, the strategy must be feasible. Yet, as explained in \S\ref{ssec:convergence}, DADP does not ensure that the coupling constraint \eqref{eqn:Pdec-6} is satisfied. To circumvent this difficulty, the thermal unit strategy is chosen in the simulation process so as to ensure feasibility of the coupling constraint, i.e.:
	\begin{equation} \label{eqn:SatisfyDemand}
		\va{u}_t^3 = \va{d}_t - \left(\va{u}_t^1+\va{u}_t^2\right).
	\end{equation}
	That is, DADP returns three strategies, for each of the hydraulic units and for the thermal unit. However, we use relation \eqref{eqn:SatisfyDemand} for the thermal strategy during simulations in order to ensure demand satisfaction and give an estimation of the cost of the DADP strategy.
\end{rem}

We run the algorithm for~20 iterations and depict its behaviour in Figure \ref{fig:DualValue}.
\begin{figure}[tb]
\begin{center}
	\includegraphics{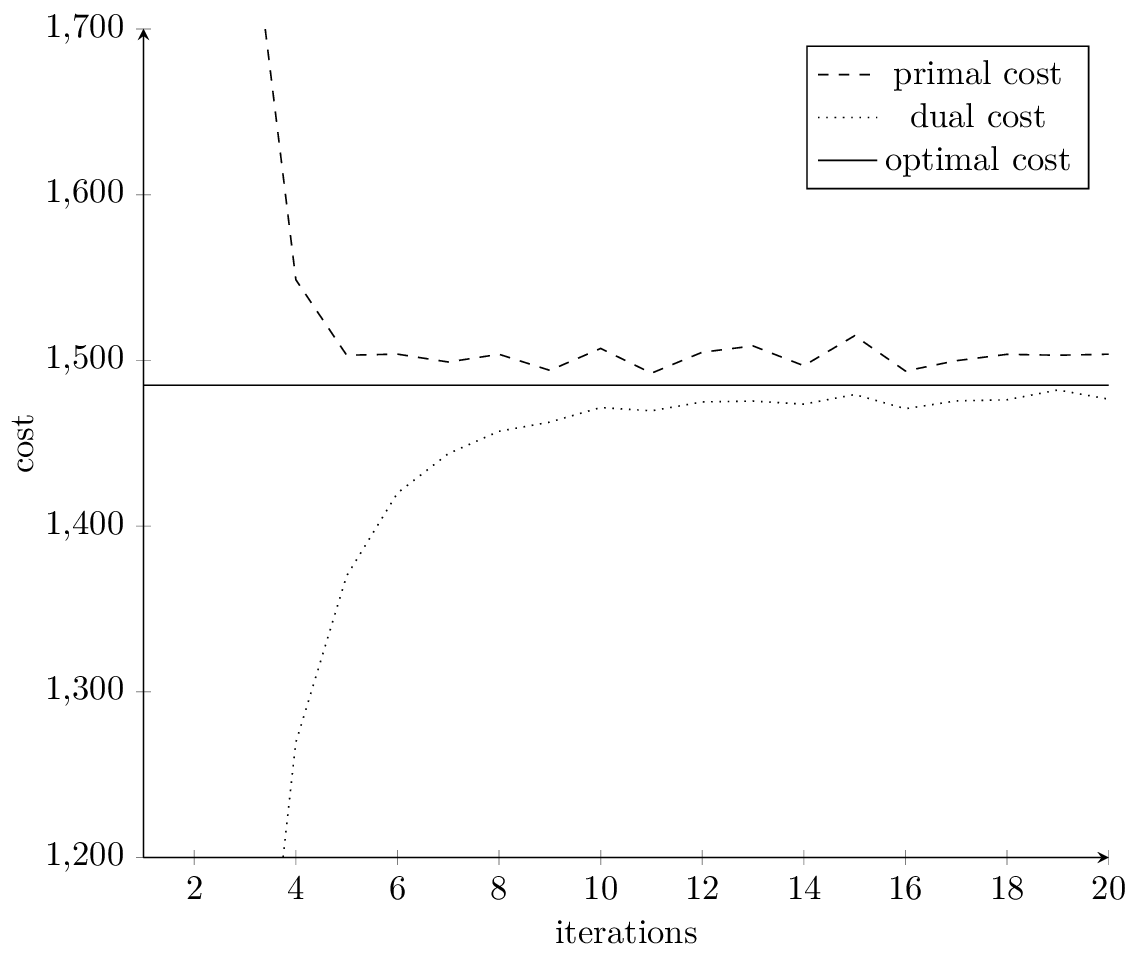}
\end{center}
\caption{\label{fig:DualValue}Primal, dual and optimal costs with respect to the number of iterations}
\end{figure}
We draw the dual cost~(evaluation of the dual function with the current strategy) and the primal cost~(the one with all constraints satisfied) at each iteration. Each point of the primal and dual curves is computed by Monte Carlo simulation over~500 scenarios. We observe the regular increase of the dual function, as expected, and the decrease of the primal function. The distance between the primal and dual costs is an upper bound for the distance to the optimal value that graphically, in this case, seems quite tight.

Moreover, the GAM toolkit used to compute the conditional expectations of the form~$\espcond{\va{\lambda}_t}{\va{d}_t}$ returns that the deviance, i.e. the quality of the explanation of~$\va{\lambda}_t$ by~$\va{d}_t$ is~98.5\%. This indicates that the marginal cost of the system is almost perfectly explained by the time variable and the power demand. Otherwise stated, using~$\espcond{\va{\lambda}_t}{\va{d}_t}$ instead of using~$\va{\lambda}_t$ within Problem~\eqref{eqn:PTest} does not alter too much the quality of the solution.

\subsection{A larger-scale SOC problem} \label{ssec:expe2}

We now apply DADP on a real-life power management problem, inspired by a case encountered at EDF, which is the major European power producer. We do not give the exact order of magnitude for costs and productions because of confidentiality issues. We consider~:
\begin{itemize}
	\item a power demand on a single node~(we neglect network issues) at each instant of a finite time horizon of~163 weeks~(one time step per week);
	\item 7~(hydraulic) stocks which are in fact aggregations of many smaller stocks;
	\item 122 other~(thermal) power units with no stock constraints.
\end{itemize}
All the thermal power units are aggregated so that the thermal cost~$\va{c}_t$ at each time~$t$ only depends on the total thermal production~$\va{u}_t^\text{th}$ and forms a quadratic cost. We note~$\va{c}_t$ using bold letters, which means that this thermal cost is random, because of the breakdowns that may happen on thermal power plants.

The problem reads:
\begin{subequations} \label{eqn:Pedf}
\begin{align}
	\min_{\va{x}, \va{u}} \quad& \esper{\sum_{t=0}^{T-1} \va{c}_t\left(\va{u}_t^\text{th}\right)}, \label{eqn:Pedf-1} \\
\intertext{subject to hydraulic stock dynamics~:}
	& \va{x}_0^i = x_0^i, \qquad \forall i=1,\dots,7, \\
	& \va{x}_{t+1}^i = \va{x}_t^i - \va{u}_t^i + \va{a}_{t}^i, \qquad \forall i=1,\dots,7, \forall t=0, \dots, T-1, \label{eqn:Pedf-3} \\
\intertext{power demand constraints~:}
	& \sum_{i=1}^7 \va{u}_t^i + \va{u}_t^\text{th} = \va{d}_t, \qquad \forall t=0, \dots, T-1, \label{eqn:Pedf-4} \\
\intertext{bound constraints on stocks and controls~:}
	& \underline{u}_t^\text{th} \leq \va{u}_t^\text{th} \leq \overline{u}_t^\text{th}, \qquad \forall t=0, \dots, T-1, \\
	& \underline{u}_t^i \leq \va{u}_t^i \leq \overline{u}_t^i, \qquad \forall i=1,\dots,7, \forall t=0, \dots, T-1, \\
	& \underline{x}_t^i \leq \va{x}_t^i \leq \overline{x}_t^i, \qquad \forall i=1,\dots,7, \forall t=0, \dots, T, \label{eqn:Pedf-7} \\
\intertext{and non-anticipativity constraints~:}
	& \va{u}_t^i \text{ is } \left(\va{w}_0, \dots, \va{w}_t\right)\text{-measurable}, \qquad \forall i=1,\dots,7, \forall t=0, \dots, T-1, \\
	& \va{u}_t^\text{th} \text{ is } \left(\va{w}_0, \dots, \va{w}_t\right)\text{-measurable}, \qquad \forall t=0, \dots, T-1,
\end{align}
\end{subequations}
with~$\va{w}_t \defegal (\va{a}_t, \va{c}_t, \va{d}_t)$ being the set of all noises that affect the system at time~$t$.

Because we consider~7 stocks, we are unable to use DP directly on this problem. In order to obtain a reference point, we use an aggregation method introduced by \citet{Turgeon80} and currently in use at EDF. This numerical method is known to be especially well-suited for the problem under consideration. It consists in solving~$N$ subproblems~(7 in our case) by 2-dimensional DP, each subproblem relying on a particular power unit, instead of one~$N$-dimensional DP problem. The idea is, for every unit, to look for strategies that depend on the stock of the unit and on an aggregation of the remaining stocks.

We then make use of DADP using three different choices for the information variable~$\va{y}_t$.
\begin{itemize}
	\item In the first setting, we replace the price at each time step by its expectation. In other words, we explain the price only by the time variable~$t$. According to Proposition~\ref{prop:convergence}, we are in fact solving Problem~\eqref{eqn:Pedf} with constraint~\eqref{eqn:Pedf-4} replaced by its expectation. Then we are able to solve each subproblem~$i$ by DP in dimension~1~(the stock variable of unit~$i$) and we obtain strategies that depend, for each unit~$i$ and each instant~$t$, on the stock~$\va{x}_t^i$ and the inflow~$\va{a}_t^i$.
	\item In the second setting, we replace the price at each time step by its conditional expectation with respect to the power demand. Put differently, we explain the price by time and demand. We still have to solve a~1-dimensional DP equation and we obtain for each instant~$t$ a strategy that depends on~$\va{x}_t^i$,~$\va{a}_t^i$ and~$\va{d}_t$.
	\item In the third setting, we replace the price at each instant by its conditional expectation with respect to the power demand and the thermal availability\footnote{The thermal availability is a scalar variable computed out of the thermal cost function~$\va{c}_t$. It gives insight on how tense the thermal generation mix is.}~$\overline{\va{p}}_t$. We then obtain a strategy that depends, for every unit~$i$ and every instant~$t$, on~$\va{x}_t^i$,~$\va{a}_t^i$,~$\va{d}_t$ and~$\overline{\va{p}}_t$.
\end{itemize}

The behaviour of the algorithm in the second setting is depicted in Figure~\ref{fig:exp2_couts}.
\begin{figure}[tb]
\begin{center}
	\includegraphics{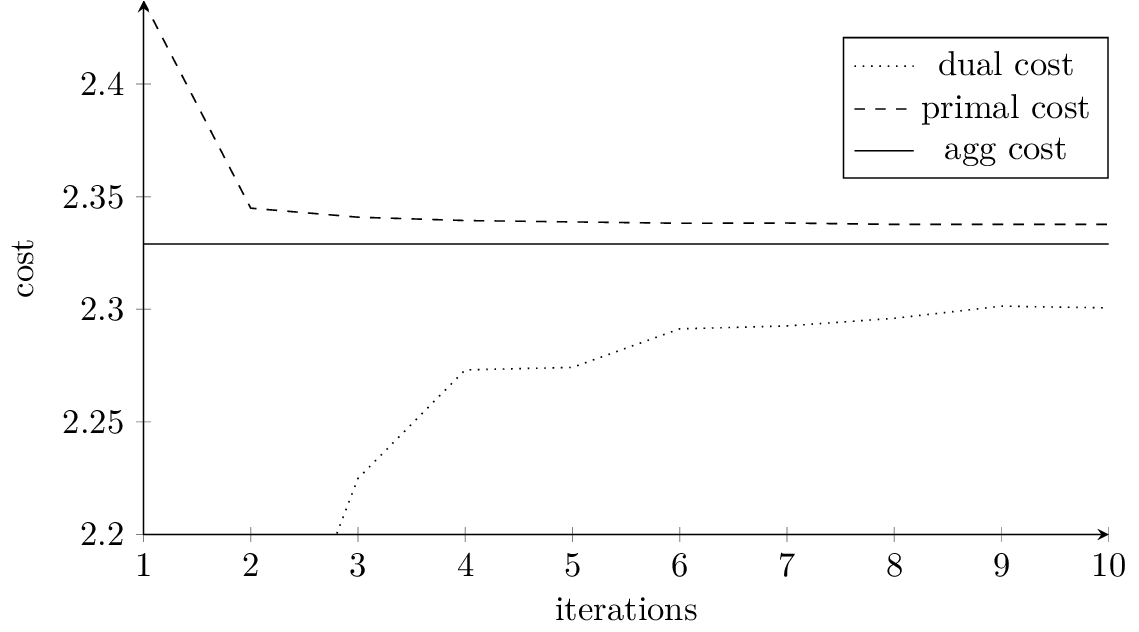}
\end{center}
\caption{\label{fig:exp2_couts}Primal and dual costs along with iterations compared to the aggregation method}
\end{figure}
We observe the increase of the dual value and the decrease of the primal value, the latter value stabilizing rapidly to a value close to the one of the aggregation method. Even though we are aware that only 10~iterations is generally much too less for this kind of primal-dual algorithm, it seems like the primal cost does not evolve significantly after~10 iterations.

In order to compare the three settings, we simulate the corresponding strategies\footnote{As in the previous example, the thermal unit strategy is chosen so as to ensure feasibility of the coupling constraint~(see Remark~\ref{rem:feasibility}).} on a large set of i.i.d. noise scenarios and compute both the mean cost and confidence interval for each strategy. The results are presented in Table~\ref{tab:Resultats}.
\begin{table}[ht]
\begin{center}
\begin{tabular}{c|c|c|c}
	 & Mean cost & $CI_{95\%}$ & Deviance \\
	\hline
	First setting & $2.363$ & $1.3 \cdot 10^{-2}$ & 50.0\% \\
	\hline
	Second setting & $2.340$ & $1.3 \cdot 10^{-2}$ & 82.4\% \\
	\hline
	Third setting & $2.338$ & $1.3 \cdot 10^{-2}$ & 86.1\%
\end{tabular}
\end{center}
\caption{\label{tab:Resultats}Results for DADP}
\end{table}
The ``Deviance'' column gives the deviance indicator returned by the GAM procedure for the estimation of the conditional expectation of the price with respect to the information variable. We observe that the DADP strategy still benefits from a good choice for the information variable~$\va{y}_t$: it appears from the mean costs comparison that adding information within the estimator improves the quality of the estimation. The mean costs differences are however not so easy to compare for the two last experiments, because the confidence interval is too large compared to the cost values. Thus we compute for each scenario the gap between costs obtained by two different strategies and draw in Figure~\ref{fig:Pdf}
\begin{figure}[tb]
\begin{center}
	\includegraphics{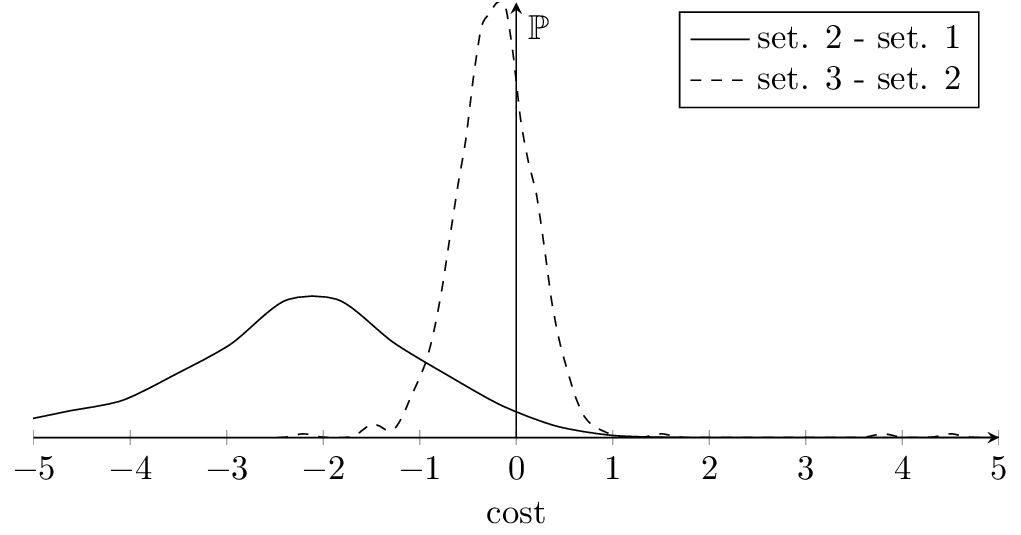}
\end{center}
\caption{\label{fig:Pdf}Distribution of cost differences between settings of DADP}
\end{figure}
the associated probability distributions. It becomes clearer that adding the thermal availability in the information variable improves the strategy: the major part of the probability weight when comparing settings~2 and~3 is negative.

As a last point, let us numerically verify that Proposition~\ref{prop:convergence} holds in our example, for instance in the first setting. Remember that, in this case, our algorithm aims at satisfying the coupling constraint only in expectation. We draw in Figure~\ref{fig:DistGap} the probability distribution of the production/demand gap at several iterations.
\begin{figure}[tb]
\begin{center}
	\includegraphics{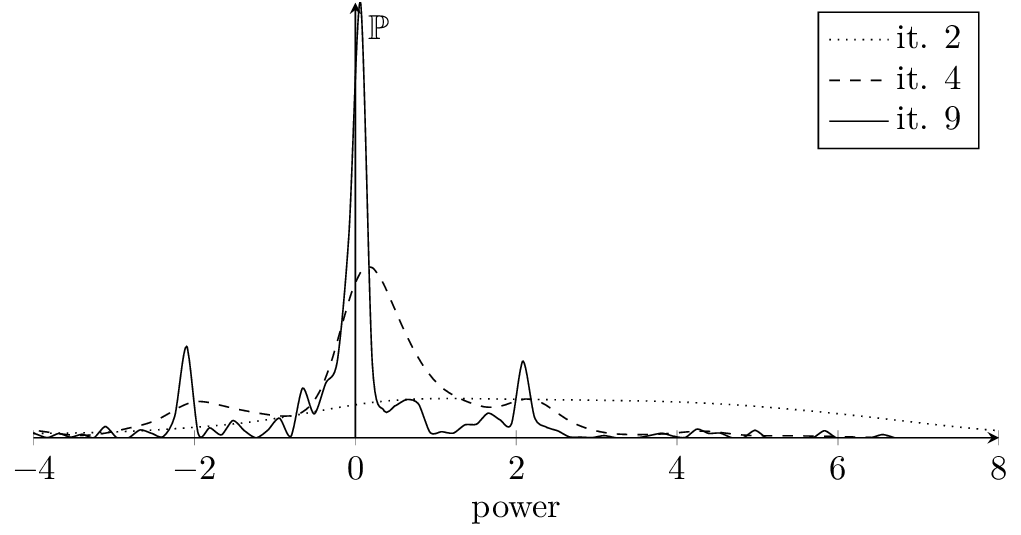}
\end{center}
\caption{\label{fig:DistGap}Distribution of the production/demand gap for a given time step}
\end{figure}
We observe that, along with iterations, the distribution of this gap becomes symmetric with respect to~$0$, the corresponding expectation hence being equal to zero.

\section*{Conclusion}

We presented an original algorithm for solving a certain kind of large-scale stochastic optimal control problems. It is based on an approximate Lagrangian decomposition: the Lagrange multiplier, which is a stochastic process in this context, is projected using a conditional expectation with respect to another stochastic process called the information process. This information process is chosen a priori and, when it has a limited memory, the solving of subproblems becomes tractable. We give theoretical results concerning the convergence of the algorithm and show how it actually solves an approximate problem, whose relation with the original problem is driven by the choice of information variable. Finally, we show on two numerical examples the efficiency of the approach.

Future works will be concerned with the application of this algorithm to more general problem structures, like chained subsystems or networks.

\appendix

\section{Duality in convex optimization} \label{app:uzawa}

The results presented here come from the paper by \citet{Cohen80}. Let~$\espacef{U}$ and~$\Lambda$ be Hilbert spaces\footnote{These results can be generalized to Banach spaces~\citep[see][]{EkelandTemam92}, but this is not necessary for our purpose.}, and~$\espacef{U}\ad$ and~$\Lambda\ad$ be subsets of~$\espacef{U}$ and~$\Lambda$~(respectively). Moreover, let us define a function~$L: \espacef{U} \times \Lambda \rightarrow \bbR$. We describe here the relations that link the so-called primal problem:
\begin{equation} \label{eqn:app_P}
	\inf_{u \in \espacef{U}\ad} \sup_{\lambda \in \Lambda\ad} L\left(u, \lambda\right),
\end{equation}
to its dual counterpart:
\begin{equation*}
	\sup_{\lambda \in \Lambda\ad} \inf_{u \in \espacef{U}\ad} L\left(u, \lambda\right).
\end{equation*}
$\espacef{U}$ is called the primal space while~$\Lambda$ is called the dual one.
\begin{dfn}[Saddle point]
	A pair~$(\overline{u}, \overline{\lambda}) \in \espacef{U}\ad \times \Lambda\ad$ is called a saddle point of~$L$ on~$\espacef{U}\ad \times \Lambda\ad$ if:
	\begin{equation*}
		L\left(\overline{u}, \lambda\right) \leq L\left(\overline{u}, \overline{\lambda}\right) \leq L\left(u, \overline{\lambda}\right), \qquad \forall u \in \espacef{U}\ad, \forall \lambda \in \Lambda\ad.
	\end{equation*}
\end{dfn}

Let us now concentrate on the case where function~$L$ corresponds to the Lagrangian of an optimization problem:
\begin{equation*}
	L\left(u, \lambda\right) = J\left(u\right) + \left\langle \lambda, g\left(u\right) \right\rangle.
\end{equation*}
The Uzawa algorithm is defined as follows. Take an initial value~$\lambda_0 \in \Lambda\ad$. At each iteration~$n \geq 0$, compute~$u_n$ by minimizing~$J\left(u\right) + \left\langle \lambda_n, g\left(u\right) \right\rangle$, and update~$\lambda_n$ using the following rule:
\begin{equation*}
	\lambda_{n+1} = \Pi_{\Lambda\ad}\left( \lambda_n + \rho_n g\left(u_n\right) \right),
\end{equation*}
with~$\rho_n$ some positive value. The following theorem gives conditions for the sequence~$(u_n)_{n \geq 0}$ to converge to the optimum of Problem~\eqref{eqn:app_P}.

\begin{thm}[\citealp{Cohen80}, Theorem~6.1] \label{thm:cohen}
	If:
	\begin{enumerate}
		\item $J$ is convex, lower semi-continuous, G\^ateaux differentiable,
		\item $J$ is $\alpha$-strongly convex,
		\item $g$ is linear and $c$-Lipschitz continuous,
		\item\label{item:saddle} $L$ has at least a saddle point~$(\overline{u}, \overline{\lambda})$,
		\item the step-size~$\rho$ of the algorithm is such that~$0 < \rho < 2 \frac{\alpha}{c^2}$,
	\end{enumerate}
	then:
	\begin{enumerate}
		\item $\overline{u}$ is unique and is a solution of Problem~\eqref{eqn:app_P},
		\item Uzawa's algorithm converges in the sense that~:
		\begin{equation*}
			u_n \underset{n\to+\infty}{\longrightarrow} \overline{u} \text{ in } \espacef{U},
		\end{equation*}
		\item the sequence~$(\lambda_n)_{n \geq 0}$ is bounded and every cluster point~$\overline{\lambda}$ in the weak topology is such that~$(\overline{u}, \overline{\lambda})$ is a saddle point of~$L$.
	\end{enumerate}
\end{thm}
Given the other assumptions of the theorem, assumption~\eqref{item:saddle} is satisfied as long as the dualized constraint satisfies a so-called ``qualification'' condition. In addition, the latter is always satisfied for affine constraints, which is the case in our application.

\begin{comment}
Note that the strong convexity assumption in Theorem~\ref{thm:cohen} is equivalent to the strong monotonicity assumption on the gradient in Theorem~\ref{thm:uzawa}.
\end{comment}

\section{A lemma about decomposition} \label{app:decomposition}

We here depict in more details the reasons why a Stochastic Optimal Problem~(SOC) involving~$N$ independent\footnote{in a sense that is made clear in Lemma~\ref{lem:dec}} subsystems is equivalent, under certain conditions, to~$N$ problems where each one involves only one of the subsystems. Though this result may seem trivial at first sight, it is not true in general: the interested reader will find a counter example in the paper by~\citet{Cohen80b}.

\begin{lemma} \label{lem:dec}
Consider the following problem:
\begin{subequations} \label{eqn:Pdecf}
\begin{align}
	\min_{\va{x}, \va{u}} \quad& \esper{\sum_{t=0}^{T-1} \sum_{i=1}^N C_t^i\left(\va{x}_t^i, \va{u}_t^i, \va{w}^{i}_t,\va{z}_t\right) + \sum_{i=1}^N K^i\left(\va{x}_T^i\right)} \\
\intertext{subject to dynamics constraints:}
	&\va{x}_{t+1}^i = f_t^i\left(\va{x}_t^i, \va{u}_t^i, \va{w}^{i}_t,\va{z}_t\right), \qquad \forall t=0, \dots,  T-1, \forall i=1, \dots, N, \label{eqn:Pdecf-2} \\
	&\va{x}_0^i \text{ is given}, \qquad \forall i=1, \dots, N, \\
\intertext{as well as bound constraints:}
	&\underline{x}_t^i \leq \va{x}_t^i \leq \overline{x}_t^i, \qquad \forall t=0, \dots, T, \forall i=1, \dots, N,\label{eqn:Pdecf-3} \\
	&\underline{u}_t^i \leq \va{u}_t^i \leq \overline{u}_t^i, \qquad \forall t=0, \dots, T-1, \forall i=1, \dots, N, \label{eqn:Pdecf-4}\\
\intertext{and the non-anticipativity constraint:}
	&\va{u}_t^i \text{ is } \trib_t\text{-measurable}, \qquad \forall t=0, \dots, T-1, \forall i=1, \dots, N, \label{eqn:Pdecf-7}
\end{align}
\end{subequations}
where \( \trib_t\) is the \(\sigma\)-algebra generated by the random variables \(\{\va{w}^{i}_s,\va{z}_s\}\) for \(i=1,\dots,N\) and \(s=0,\dots,t\).
We assume that:
\begin{itemize}
  \item the \(\va{w}_{\cdot}^i\)'s and \(\va{z}_{\cdot}\) are all white noise processes,
  \item that \(\va{w}_{t}^i\) is not necessarily independent from \(\va{w}_{t}^j\) for \(j\neq i\) nor from \(\va{z}_t\).
  %\footnote{If they were independent, then obviously the problem would split up into \(N\) independent subproblems.}
\end{itemize}
Then, the optimal feedback solution is \emph{partially decentralized}, that is, each optimal decision \(\va{u}_t^i\), that may a priori depend on the whole \(\va{x}_t\) and the whole \(\va{w}_t\) and \(\va{z}_{t}\) according to \eqref{eqn:Pdecf-7}, indeed only depends on \((\va{x}_t^i,\va{w}_t^i,\va{z}_{t})\); the Bellman function \(V_{t}(\va{x}_t)\) is additive (\(V_{t}(\va{x}_t)=\sum_{i=1}^{N}V_{t}^{i}(\va{x}_t^{i})\)) and the optimal solution only involves the \emph{marginal} probability laws of the pairs \((\va{w}_t^i,\va{z}_{t})\) but not the \emph{joint} probability laws of the pairs \((\va{w}_t,\va{z}_{t})\).
\end{lemma}
\begin{proof}
The proof is by induction over time. The statement that \(V\) is additive is true at the final time~\(T\) since the final cost \(K\) is additive. Assume this is true from \(T\) to \(t+1\) (backward). The Bellman equation at \(t\) reads:
\begin{equation*}
	V_{t}(x)=\bgesp{\min_{u}\sum_{i=1}^{N}C_{t}^{i}(x^{i},u^{i},\va{w}_{t}^i,\va{z}_{t}) + \sum_{i=1}^{N}V_{t+1}^{i}\big(f_{t}^{i}(x^{i},u^{i},\va{w}_{t}^i,\va{z}_{t})\big)}\,,
\end{equation*}
in which 
\begin{itemize}
  \item the minimization operation is done over an expression is which \(x\), \(\va{z}_t\) and \(\va{w}_t^i\) are fixed (hazard-decision scheme) and the arg\,min in \(u\) parametrically depends on those values (which yields the optimal feedback function)~;
\item the minimization operation is subject to the bound constraints \eqref{eqn:Pdecf-4} for \(u^{i}\) and \eqref{eqn:Pdecf-3} for \(f_{t}^{i}(x^{i},u^{i}, \va{w}_t^i,\va{z}_t)\)~;
\item the expectation concerns random variables \((\va{w}_{t},\va{z}_{t})\) whereas \(x\) is still fixed (\(\va{X}_{t}\) and \((\va{w}_{t},\va{z}_{t})\) are independent from each other, thus this expectation may be considered as a conditional expectation knowing that \(\va{x}_{t}=x\)): this yields a function of \(x\), namely \(V_{t}(\cdot)\). 
\end{itemize}

Now observe that, at the minimization stage, each \(u^{i}\) is involved into a separate expression depending only on \(x^{i}\), \(\va{w}_t^{i}\) and \(\va{z}_t\) subject also to independent constraints, hence the claimed partially decentralized optimal feedback. Then, at the outer expectation stage, we get a sum of functions of \(x^i\) and \((\va{w}_{t}^i,\va{z}_{t})\): thus only the marginal probability law of each pair \((\va{w}_{t}^i,\va{z}_{t})\) is involved in the expectation of the corresponding term in this sum, and the result is an additive function of the \(x^{i}\), which completes the proof by induction.
\end{proof}

Let us now comment some particular cases.
\begin{itemize}
  \item If \(\va{z}_{t}\) is absent and if \(\va{w}^{i}_{\cdot}\) and  \(\va{w}^{j}_{\cdot}\) are independent whenever \(j\neq i\), then the overall problem is obviously made up of \(N\) independent subproblems; the optimal feedbacks are fully decentralized (that is \(\va{u}^{i}\) is in closed loop on \((\va{x}^{i},\va{w}^{i})\)), and the optimal controls \(\va{u}^{i}_{\cdot}\) and \(\va{u}^{j}_{\cdot}\) are also independent random variables whenever \(j\neq i\).
  \item If we drop the independency assumption about \(\va{w}^{i}_{\cdot}\) and  \(\va{w}^{j}_{\cdot}\), then the same subproblems still provide the overall problem  solution with decentralized feedbacks, but \(\va{u}^{i}_{\cdot}\) and \(\va{u}^{j}_{\cdot}\) are no longer independent. 
  \item Another ``extreme'' situation is when only the ``shared'' noise \(\va{z}\) is present in all subsystems (the \(\va{w}^{i}\)'s are supposed absent for the sake of clarity but now, \(\va{z}\) may be thought as the concatenation of all the \(\va{w}^{i}\)'s). The conclusions of the lemma are of course still valid, that is, the Bellman function is still additive and each term of this sum can be calculated in a separate subproblem, yielding a feedback on \((\va{x}^{i},\va{z})\). However the price to be payed for the presence of this shared random variable is that, first, the minimization operation in the Bellman function is parametrized by both \(x^{i}\) and \(\va{z}_t\), which may be costly if \(\va{z}_t\) is of large dimension, and, second, the outer expectation in this Bellman equation involves a multiple integral over that vector \(\va{z}_t\), which may also be costly. 
\end{itemize}

\end{document}